\documentclass{amsart} 
\usepackage{amssymb,latexsym}
\usepackage{amsmath}
\usepackage{graphicx} 
\usepackage{enumerate}
\usepackage{shortCongBook3}
\usepackage{gensymb}

\theoremstyle{plain}
\newtheorem{theorem}{Theorem}
\newtheorem{lemma}[theorem]{Lemma}

\theoremstyle{definition}
\newtheorem{definition}[theorem]{Definition}

\begin{document}
\title[An open problem on congruences of finite lattices]{An open problem\\on congruences of finite lattices\\
in pictures}
\author[G.\ Gr\"atzer]{George Gr\"atzer}
\email{gratzer@mac.com}
\urladdr{http://server.maths.umanitoba.ca/homepages/gratzer/}
\address{University of Manitoba}
\date{\today}

\begin{abstract} 
Let $L$ be a planar semimodular lattice. 
We call $L$ \emph{slim}, 
if it has no $\mthree$ sublattice.
Let us define an \emph{SPS lattice} as a slim, planar, semimodular lattice $L$. 

In 2016, I proved a property of congruences of SPS lattices (Two-cover Property) 
and raised the problem of characterizing them.

Since then, more than 50 papers have been published contributing to this problem.
In this survey, I provide an overview of this field with major contributions by 
G\'abor Cz\'edli.
\end{abstract}

\maketitle

\section{Introduction}\label{S:Introduction}

Paul Erd\H{o}s defined a mathematician as a machine, where the input is (espresso) 
coffee and the output is theorems and conjectures. 
He thought that raising problems was as important as proving theorems.

Here is an illustration.  My book  \cite{GLT},  \emph{General Lattice Theory}, was published in 1978.
A year after its publication, I~wrote an article (see \cite{gG80}), \emph{1979 problem update}.
The article has 21 references affecting 40 problems proposed in the book. 
This book and its second and third editions (see \cite{GLT2} and \cite{LTF}) have 850 references in the Mathematical Reviews,
and many of them are connected to open problems these books propose.

This survey focuses on a single open problem on
congruences of finite lattices; I~proposed it in a 2016 paper (see \cite{gG17a}),
as related in my book \cite{CFL3}.

\subsection{Notation}\label{S:Notation}
We use the notation as in my book \cite{CFL3}.
Part I of this book can be freely downloaded: \texttt{arXiv:2104.06539}

\subsection{Outline}\label{S:Outline}

Birkhoff duality helps in computing in large distributive lattices;
this topic is introduced in Section~\ref{S:distributive}.
Section~\ref{S:started} recalls how this field started. 
A~crucial result  of G. Gr\"atzer, H.~Lakser, and  E.\,T. Schmidt~\cite{GLS98a} 
is stated in Section~\ref{S:semimodular}.
Section~\ref{S:main} formulates the Problem.
The first known property of congruences of finite lattices, the Two-cover Property,
is given in Section~\ref{S:Two}.

There are three major tools required to tackle the Problem,
they are described in Section~\ref{S:Swing} Sections~\ref{S:Natural}, and \ref{S:Lamps}. 
Section~\ref{S:major} has the four major properties of G.~Cz\'edli.
The 3P3C property is discussed in Section~\ref{S:3P3C}.

To illustrate how the tools are used, Section~\ref{S:proofs} sketches some proofs.
Finally, Section~\ref{S:more} formulates a related problem.


\section{Finite distributive lattices}\label{S:distributive}

Small distributive lattices are easy to compute in. Larger ones, not so much.
\emph{Birkhoff duality} (see G. Birkhoff~\cite{gB37}) provides a solution.

For a finite ordered set $P$, call $A \ci P$ a \emph{down set} if{}f $x
\in A$ and $y \leq x$ in $P$, imply that $y \in A$. For $H \ci P$, there is
a smallest down set containing $H$, namely, 
$\setm{x}{x \leq h, \ \text{for some $h \in H$}}$; we use the notation 
$\Dg H$ for this
set. If $H = \set{a}$, we write $\Dg a$ for $\Dg \set{a}$.
Let $\Down P$ denote the set of all  down sets ordered by
set inclusion. Then $\Down P$ is a finite distributive lattice.

Conversely, for  a finite distributive lattice $D$, 
let $P$ denote the ordered set of join-irreducible elements of $D$.

\begin{named}{Birkhoff duality} 
\begin{enumeratei}\hfill

\item The assignment $P \to \Down P$ is a bijection between finite ordered sets
and finite distributive lattices. It assigns $a \in P$ to $\Dg a \in  \Down P$.
\item The inverse map is $D \to \Ji D$. It assigns $a \in D$ to $\id a \ii P$.
\end{enumeratei}
\end{named}

Figure~\ref{F:ree-new1} provides an example. 
The finite distributive lattice $D$ on the left has 20 elements 
and a fairly complicated structure.
The corresponding finite ordered set $P$ on the right has only 6 elements.

\begin{figure}[htb]
\centerline{\includegraphics{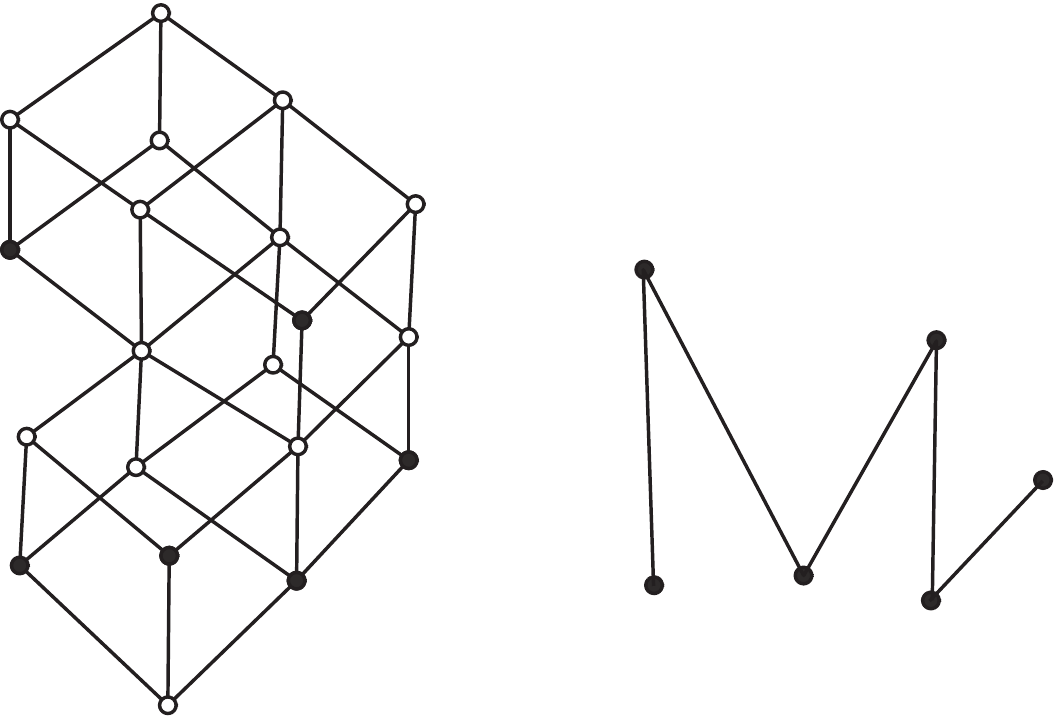}}
\caption{A finite distributive lattice and the corresponding finite ordered set}
\label{F:ree-new1}
\end{figure}

For a finite lattice $L$, the congruences form a finite distributive lattice, $\Con L$.
The corresponding finite ordered set is the ordered set of join-irreducible congruences,
$\Ji {\Con L}$.
For $a \prec b$ in $L$, we can form $\con{a,b}$,
the smallest congruence of~$L$ collapsing $a$ and~$b$.
Then $\con{a,b} \in \Ji {\Con L}$ and conversely.
  
\section{How it started}\label{S:started}

The first result in this field is in a 1942 paper, N.~Funayama and T.~Nakayama \cite{FN42}:
the congruence lattice of a lattice is distributive.
So the congruence lattice of a finite lattice is a finite distributive lattice.

The converse was discovered by R. P. Dilworth,
unpublished, \emph{circa} 1944 
(see the book, K.\,P. Bogart, R. Freese, and J.\,P.\,S. Kung, eds, \cite{BFK90}
for a discussion).

\begin{theorem}
Every finite distributive lattice $D$ can be represented
as the congruence lattice of a finite lattice $L$.
\end{theorem}

The first proof of the this result was published in 1962
(see G. Gr\"atzer and E.\,T. Schmidt \cite{GS62});
it represented $D$ with a \emph{sectionally complemented} lattice $L$.

\section{Representing with semimodular latices}\label{S:semimodular}

Can we represent every finite distributive lattice $D$
as  $\Con L$, the congruence lattice of a~finite \emph{semimodular} lattice $L$?

Or, equivalently, can we represent every finite ordered set $P$
as  $\Ji \Con L$, the ordered set of join-irreducible congruences 
of a finite \emph{semimodular} lattice $L$?

Indeed, we can, as verified in G. Gr\"atzer, H.~Lakser, and  E.\,T. Schmidt~\cite{GLS98a}.

\begin{theorem}\label{T:PS}
Every ﬁnite distributive lattice $D$ can be represented
as the congruence lattice of a ﬁnite planar semimodular lattice $L$.
\end{theorem}

It was a great surprise to us that the proof yielded a \tbf{planar lattice}.

We now provide a \emph{Proof-by-Picture}, utilizing the semimodular lattice  $\seight$,
see Figure~\ref{F:S8}, to represent the ordered set $P$ of Figure~\ref{F:smorder}
as $\Ji \Con L$ of a finite semimodular lattice~$L$

A \emph{colored lattice} is a finite lattice 
(some) of whose edges (prime intervals) are labeled 
so that if the edges $\fp$
and $\fq$ are of the same color, then $\con{\fp} = \con{\fq}$. 
These labels represent equivalence classes of edges 
defining the same join-irreducible congruence.

\begin{figure}[b!]
\centerline{\includegraphics{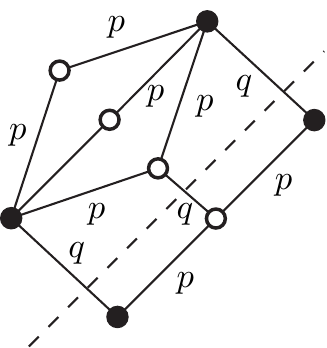}}
\caption{The semimodular lattice $\seight$, colored with $p < q$}\label{F:S8}
\end{figure}

\begin{figure}[htb]
\centerline{\includegraphics{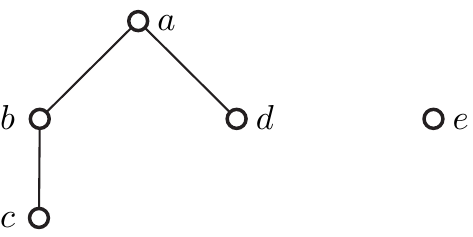}}
\caption{The ordered set $P$}\label{F:smorder}
\end{figure}

\begin{figure}[htb]
\centerline{\includegraphics{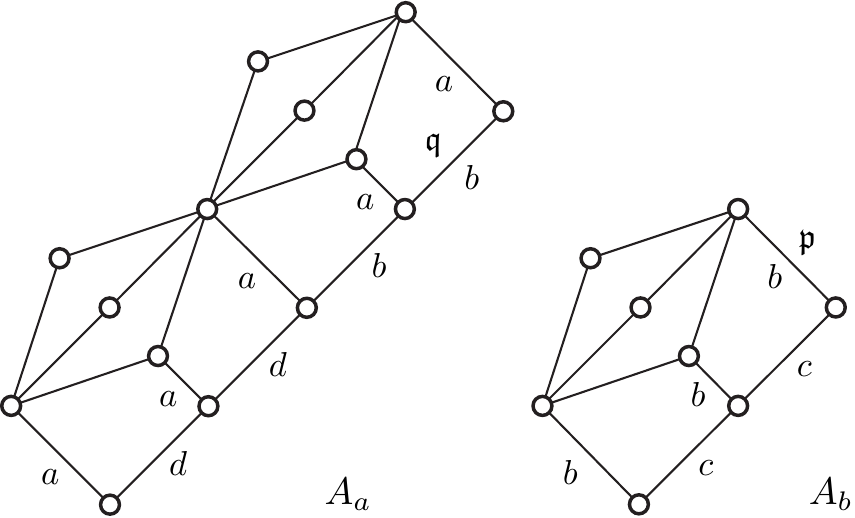}}
\caption{Two colored lattices}\label{F:AaAb}
\end{figure}

To represent $P$, we start out by
constructing the colored lattices $A_a$ and~$A_b$  (see Figure~\ref{F:AaAb}).
We obtain $A_a$ by gluing two copies of $\seight$
together, coloring it with $\set{a,b,d}$ so that $\con{a} > \con{b}$ is
accomplished in the top~$\seight $ of~$A_a$ and $\con{a} > \con{d}$ is accomplished
in the bottom $\seight $ of $A_a$. The~lattice~$A_b$ is $\seight $ colored 
by $\set{b,c}$, so~that $\con{b} > \con{c}$ in $A_b$.

Observe that the lattice $A_a$ takes care \emph{of all} $a \cov x$ orderings;
in the example, there are only two. We could do three coverings by gluing
three copies of $\seight $ together, and so on.

Form the glued sum $S$ of $A_a$ and $A_b$; all the covers of $P$ are
taken care of in~$S$. There is only one problem: $S$ is not a 
colored lattice; in this example, if $\fp$ is a prime interval of color $b$
in $S_b$ (as in Figure~\ref{F:AaAb}) and $\fq$ is a prime interval of color $b$
in $S_b$  (as in Figure~\ref{F:AaAb}), then in $S$ we have $\con \fp \mm \con \fq = \zero$.
Of course, we should have $\con \fp = \con \fq = \con b$. We accomplish this
by~extending $S$ to the lattice $L$ of Figure~\ref{F:Lnew}. In $L$, the
black-filled elements form the sublattice~$S$.

\begin{figure}[t!]
\centerline{\scalebox{.78}{\includegraphics{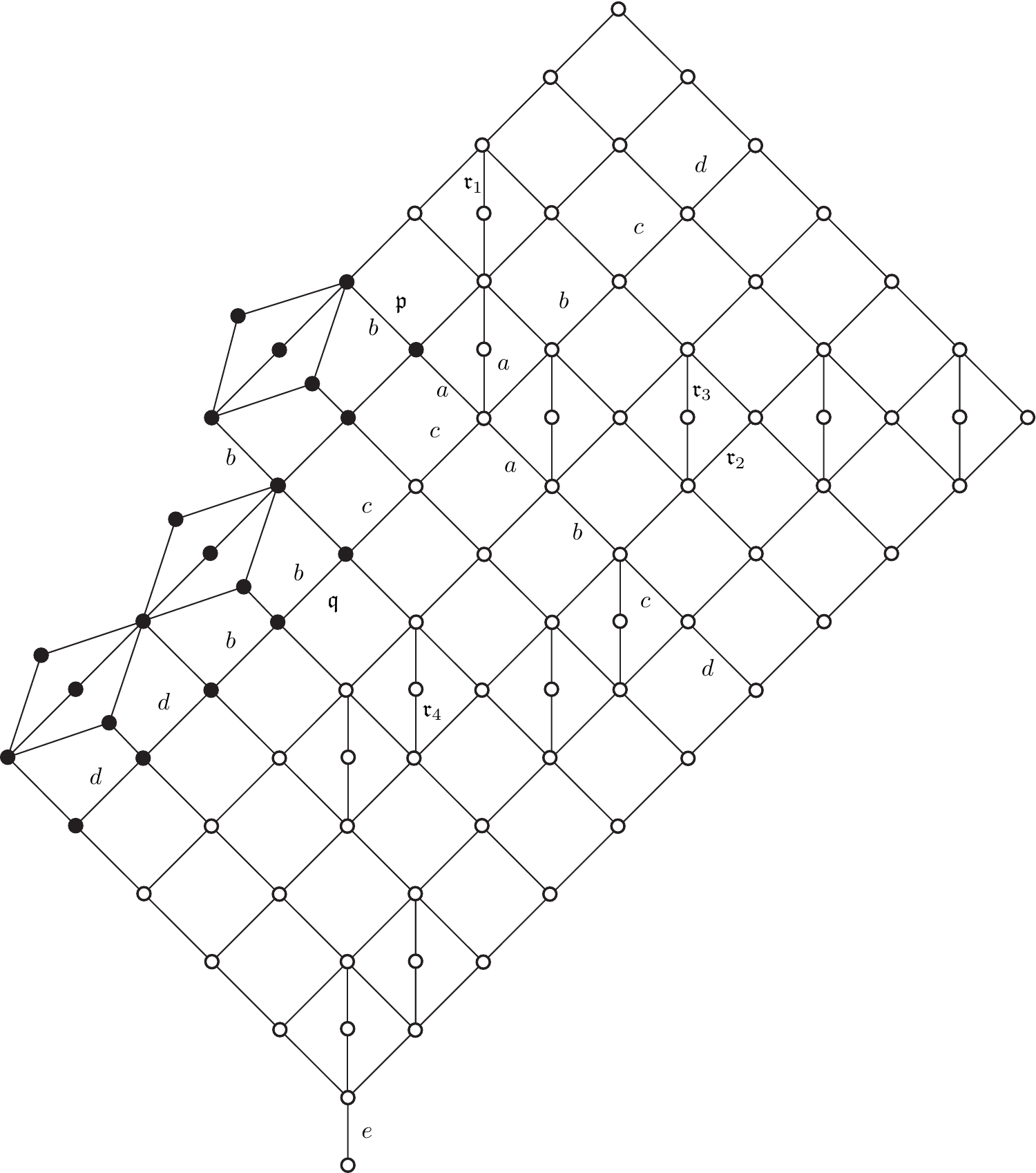}}}
\caption{The lattice $L$}\label{F:Lnew}
\end{figure}

\begin{figure}[hbt]
\centerline{\includegraphics{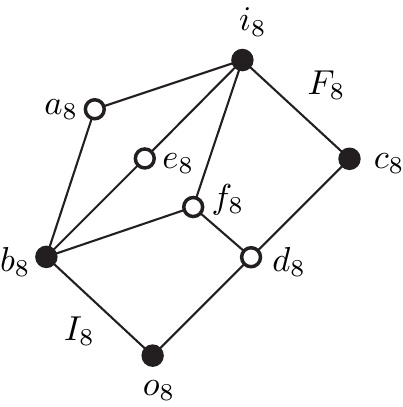}}
\caption{The semimodular gadget $\seight $, with notation}\label{F:S8labels}
\end{figure}

As you see, we extend $S$ by adding to it a distributive ``grid.'' The right
corner is $\SC 5^2$ colored by $\set{a,b,c,d}$; each of the four covering
squares colored by the same color twice are made into a cover-preserving $\SM{3}$. This
makes the coloring behave properly in the right corner. In the rest of the
lattice we do the same:  we look for a covering square colored by the same
color twice, and make it into a cover-preserving~$\SM{3}$. This makes $L$ into a colored
lattice: any two prime intervals of the same color generate the same
congruence. 

Finally, we remember the element $e \in P$. We add a ``tail'' to the lattice
and color it $e$.

The resulting lattice $L$ is planar and semimodular 
and $\Ji{\Con L} $ is isomorphic to $P$ 
with the isomorphism $x \mapsto \con{x}$ for $x \in \set{a,b,c,d,e}$.

\section{The Problem}\label{S:main}

According to Theorem~\ref{T:PS}, 
every ﬁnite distributive lattice $D$ can be represented
as the congruence lattice of a ﬁnite planar semimodular lattice $L$.
In the construction, we strongly utilize $\mthree$ sublattices. Could we do the construction without them?

Let us call a ﬁnite planar semimodular lattice $L$ \emph{slim}, if it has no $\mthree$ sublattice.
Let us define an \emph{SPS lattice} as a slim, ﬁnite, planar semimodular lattice $L$ 
without an $\mthree$ sublattice. (SPS stands for \tbf{S}lim, \tbf{P}lanar, \tbf{S}emimodular.)

In the paper \cite{gG17a}, I~proposed the following.

\begin{named}{Problem} 
Characterize the congruence lattices of SPS lattices.
\end{named}

\section{The Two-cover Property}\label{S:Two}

A finite ordered set $P$ satisfies the \emph{Two-Cover Condition}
(see my paper  \cite{gG17a}),
if any element of~$P$ has at most two covers.
Figure~\ref{F:3fork} shows the diagram of an ordered set $F_3$
failing the Two-Cover Condition. 
A finite ordered set $P$ satisfies the Two-Cover Condition
if{}f $F_3$
has no \emph{cover-preserving embedding} into $P$.

The following is the main result of my paper \cite{gG17a}.

\begin{theorem}[Two-Cover Theorem]\label{T:Two-Cover}
The ordered set of join-irreducible congruences 
of an SPS lattice $L$ has the Two-Cover Condition.
\end{theorem}
 
\begin{figure}[htb]
\centerline{\includegraphics{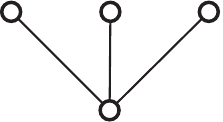}}
\caption{The ordered set $F_3$ for the Two-Cover Condition}
\label{F:3fork}
\end{figure} 

The Two-Cover Condition is a necessary condition. 
G.~Cz\'edli~\cite{gC14a} proves that it is not sufficient.  
My paper \cite{gG15c} provides a short proof.

Let us say that a finite ordered set $P$ satisfies the 
 \emph{Two-max Condition},
if~$P$ has at least two maximal elements.

In my note \cite{gG20}, I prove the following very easy result.

\begin{theorem}[Two-max Theorem]\label{T:Two-max}
\index{Two-max Theorem}%
The ordered set of join-irreducible congruences 
of a nontrivial SPS lattice $L$ has the Two-max Condition.
\end{theorem}

It follows that the three-element chain $\SC3$,
which satisfies the Two-Cover Condition, cannot be represented
as the congruence lattice of~an SPS lattice;
of course, an example of minimal size.

G.~Cz\'edli \cite{gC21} discovered four major properties 
of the congruence lattices of SPS lattices. 
We'll discuss them in Section~\ref{S:major}.
We present the sixth major property in Section~\ref{S:3P3C}, see  G. Cz\'edli 
and G.~Gr\"atzer~\cite{CG21}. 

G. Cz\'edli\index{Cz\'edli, G.}  maintains a list of papers making contributions to this problem,
52 as of this writing, see

\texttt{http://www.math.u-szeged.hu/\~{}czedli/m/listak/publ-psml.pdf}\\
They are all listed in the References.

\section{The Tools}\label{S:Tools}
  
\subsection{The Swing lemma}\label{S:Swing}

For the edges $\fp, \fq$ of an SPS lattice $L$,
we define a binary relation:
\index{Swing}%
$\fp$~\emph{swings to} $\fq$, written as $\fp \swing \fq$,
if $1_\fp = 1_\fq$, 
this element covers at least three elements,
and $0_\fq$ is neither the left-most nor the right-most element
covered by $1_\fp = 1_\fq$. 
We call the element $1_\fp = 1_\fq$
\index{Hinge}%
the \emph{hinge} of the swing.\label{ToN:swings}
If $0_\fp$ is either the left-most or the right-most element
covered by the hinge, then we call the swing \emph{external},
in formula, $\fp \exswing \fq$. 
Otherwise, the swing is \emph{internal},
in formula, $\fp \inswing \fq$.
See Figure~\ref{F:Swings} for two examples; 
in the first, the hinge covers three elements and $\fp \exswing \fq$; 
in~the second, the hinge covers five elements
and $\fp \inswing \fq$.

\begin{figure}[hbt]
\centerline{\includegraphics{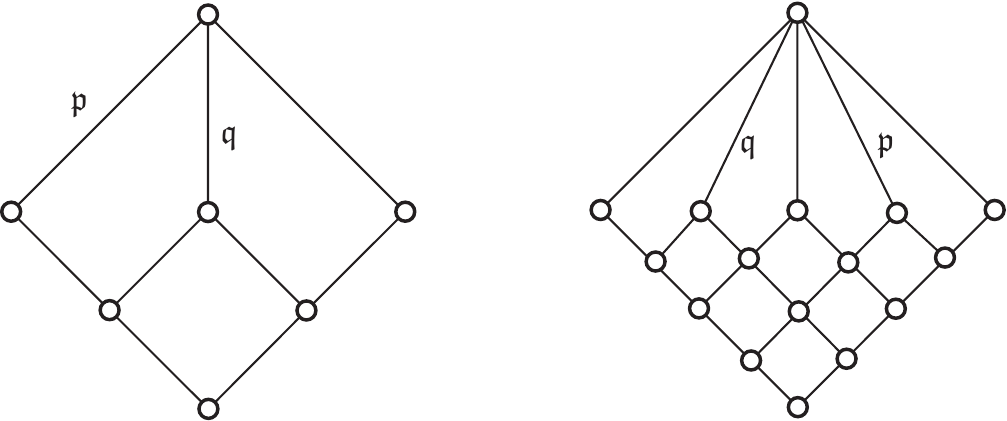}} 
\caption{Swings, $\fp \protect\swing \fq$}\label{F:Swings}
\end{figure}

\begin{lemma}[Swing Lemma]\label{L:SPSproj}\index{Swing Lemma}
Let $L$ be an SPS lattice\index{Lattice!SPS|}\index{SPS lattice|}\index{Lattice!slim, planar, semimodular (SPS)}
and let $\fp$ and $\fq$ be distinct edges in $L$. 
Then $\fq$ is collapsed by $\con{\fp}$ if{}f 
there exists a edge~$\fr$ 
and sequence of pairwise distinct edges
\begin{equation}\label{Eq:sequence}
\fr = \fr_0, \fr_1, \dots, \fr_n = \fq
\end{equation}
such that $\fp$ is up perspective to $\fr$, and 
$\fr_i$ is down perspective to or swings to $\fr_{i+1}$
for $i = 0, \dots, n-1$. 
In addition, the sequence \eqref{Eq:sequence} also satisfies 
\begin{equation}\label{E:geq}
   1_{\fr_0} \geq 1_{\fr_1} \geq \dots \geq 1_{\fr_n}.
\end{equation}
\end{lemma}

The Swing Lemma is easy to visualize using skiing analogies.
\begin{enumeratei}
\item Perspectivity up is ``climbing'' (see Figure~\ref{F:gondola}).
\item Perspectivity down is ``sliding'' (see Figure~\ref{F:downps}).
\item Swinging  is ``turning'' (see Figure~\ref{F:Swing}).
\end{enumeratei}

So we get from $\fp$ to $\fq$ by climbing once
and then alternating sliding and swinging.
In the example of Figure~\ref{F:traj}, 
we climb up from $\fp$ to $\fr = \fr_0$, 
swing from $\fr_0$ to~$\fr_1$,
slide down from $\fr_1$ to $\fr_2$,
swing from $\fr_2$ to $\fr_3$,
and finally slide down from $\fr_3$ to $\fr_4$.
\subsection{Natural diagrams and $\E C_1$-diagrams}\label{S:Natural}

As in my joint paper with E.~Knapp \cite{GKn07},
we call the planar semimodular lattice $L$ \emph{rectangular},
if its left boundary chain has exactly one doubly-irreducible element,~$\cornl L$,
and its right boundary chain has exactly one doubly-irreducible element,
$\cornr D$, and these elements are complementary,\label{ToN:cornr}
that~is,
\begin{align*}
\cornl D \jj \cornr D&=1,\\
\cornl D \mm \cornr D&=0.
\end{align*}
We call a slim rectangular lattice an \emph{SR lattice}.

In this section, we consider nice diagrams for SR lattices.

\begin{figure}[t!]
\centerline{\includegraphics[scale = .6]{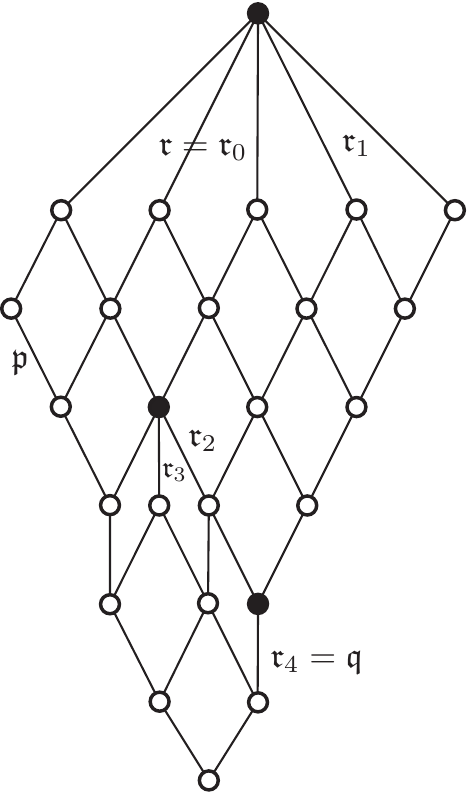}}
\caption{Illustrating the Swing Lemma}\label{F:traj}

\bigskip

\bigskip

\centerline{\includegraphics[scale = .26]{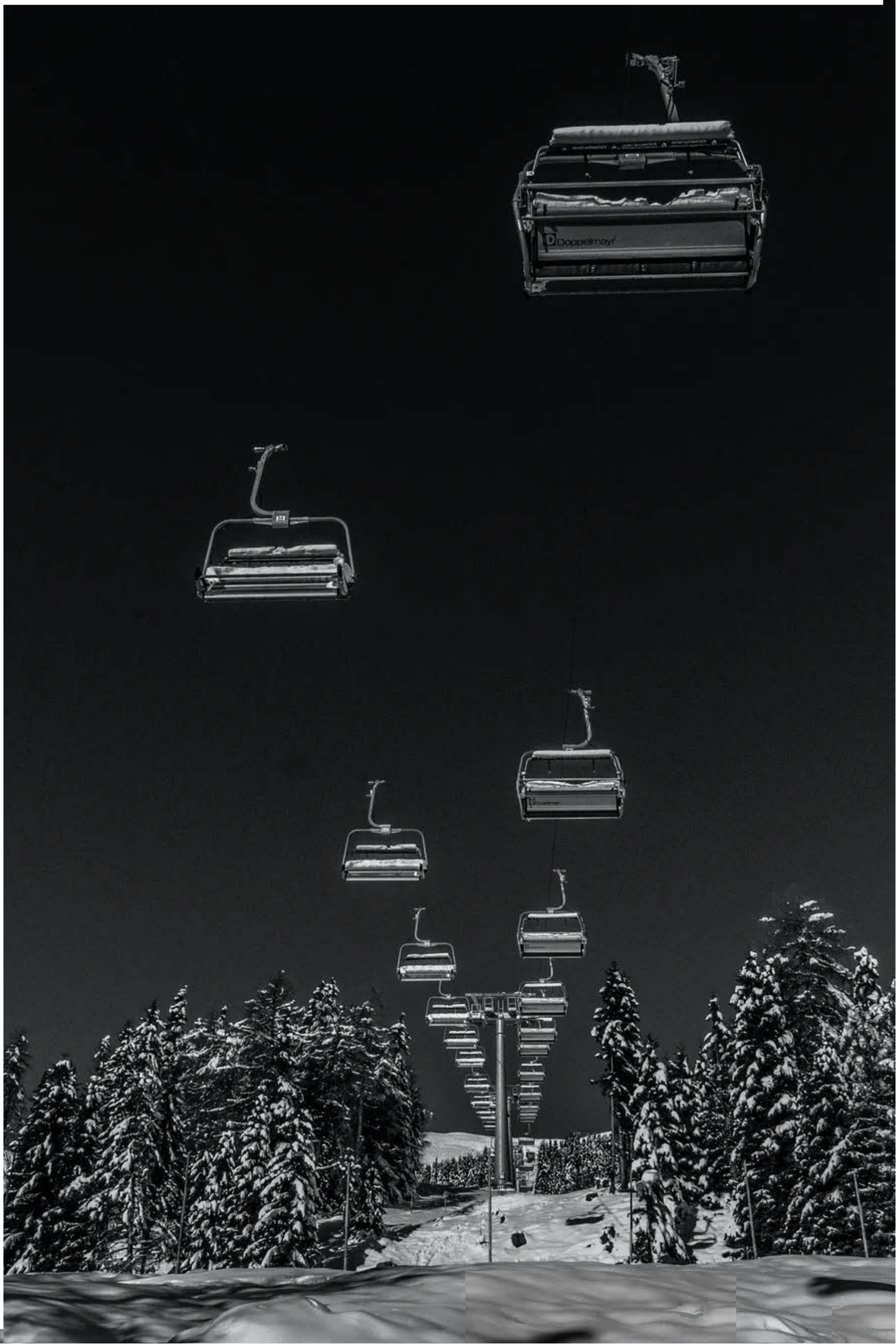} \q\q\q\q\q\q\q\q\q\q\ \ 
\includegraphics[scale = .6]{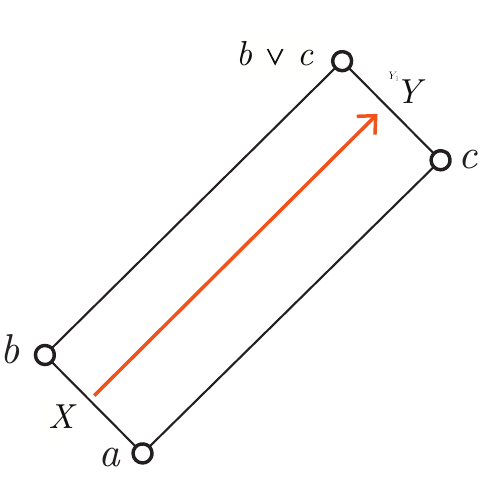}}
\caption{Take the gondola up---up-perspectivity}
\label{F:gondola}

\bigskip

\bigskip

\centerline{\includegraphics[scale = .26]{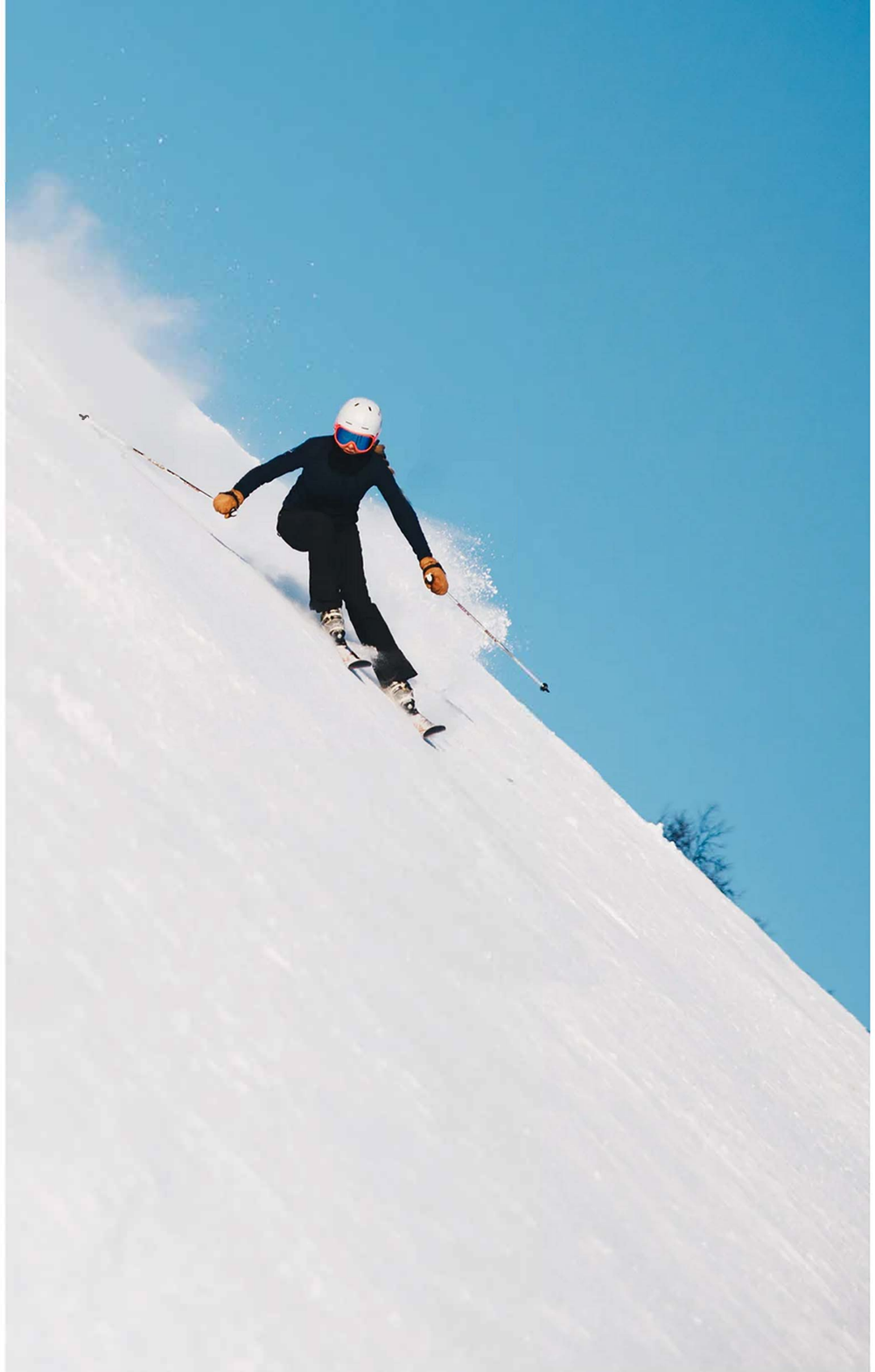}\q\q\q\q\q\q \q\q\q\ 
\includegraphics[scale = .6]{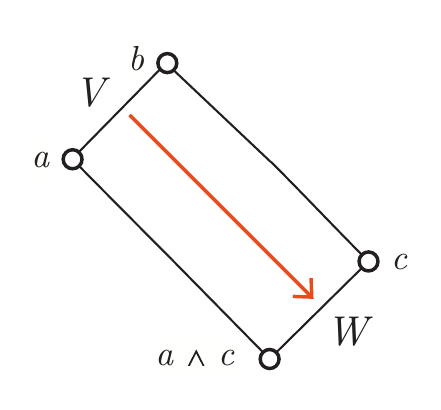}}
\caption{Ski down---down-perspectivity}
\label{F:downps}

\bigskip

\bigskip

\centerline{\includegraphics[scale = .25]{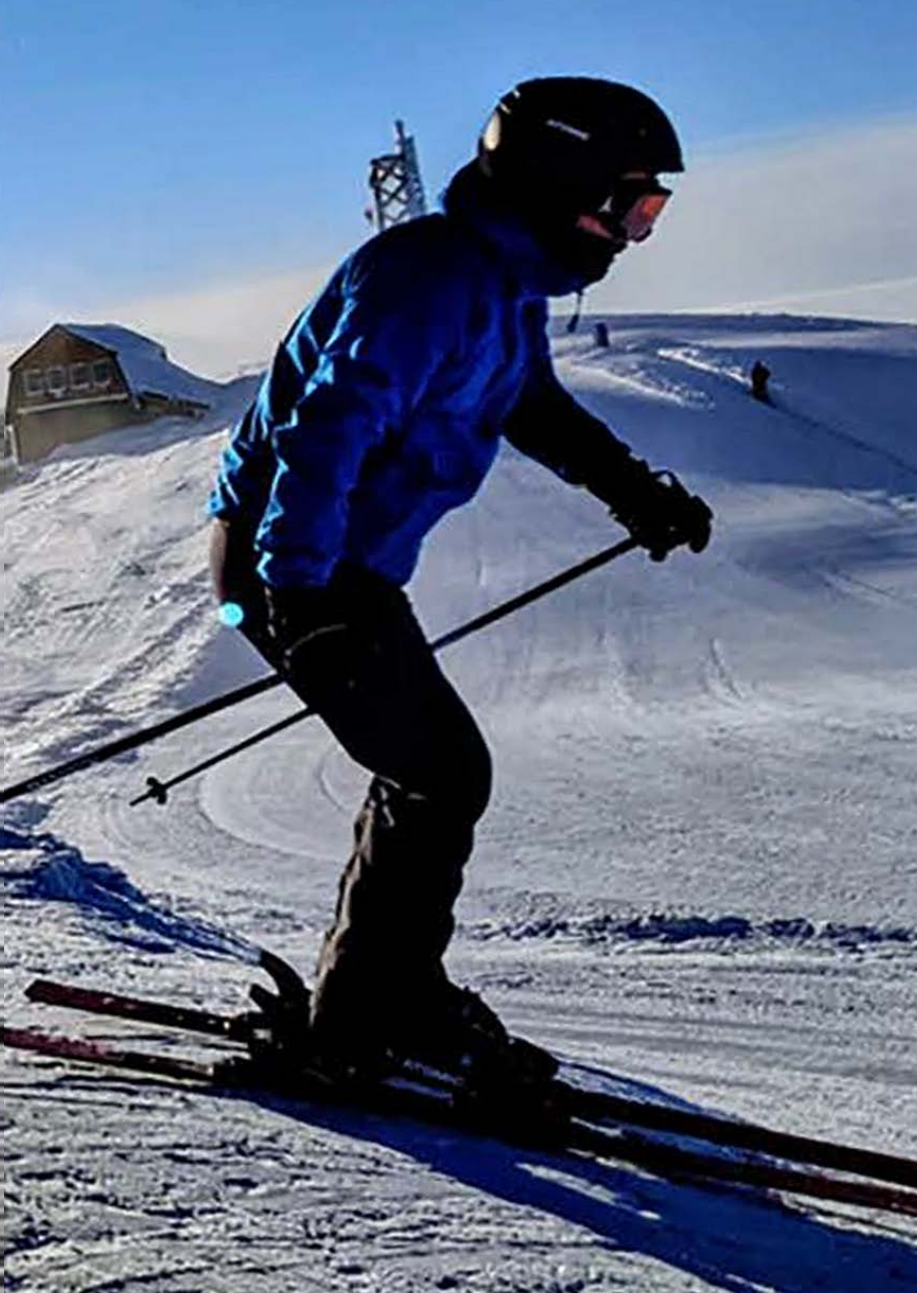}\q\q\q\q\q\q\includegraphics[scale = 1]{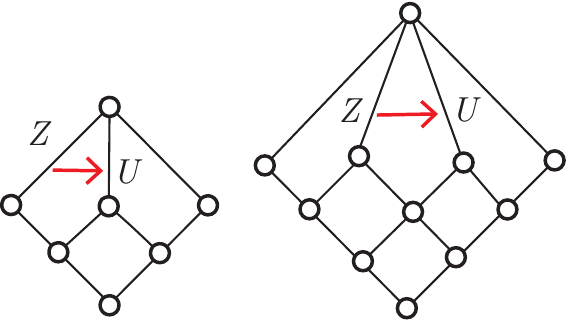} }
\caption{Turn---Swing}
\label{F:Swing}
\end{figure}

\subsection*{Natural diagrams for SR lattice}

We follow my joint paper with E. Knapp~\cite{GKn10} published in 2010.
For an SR lattice $L$,
let $\Chl L$ be the lower left and $\Chr L$ the lower right boundary chain of $L$, respectively.

We regard the grid $G = \Chl L \times \Chr L$ as a planar lattice, 
with $\Chl L = \Chl G$ and $\Chr L = \Chr G$. 
Then the map
\[
   \gy \colon  x \mapsto (x \mm \cornl L, x \mm \cornr L)
\]
is a meet-embedding of $L$ into $G$;
the map $\gy$ also preserves the bounds.
Therefore, the image of $L$ under $\gy$ in $G$ is a diagram of $L$;
we call it the \emph{natural diagram} representing $L$. 

\begin{figure}[t!]
\centerline{\includegraphics[scale=.7]{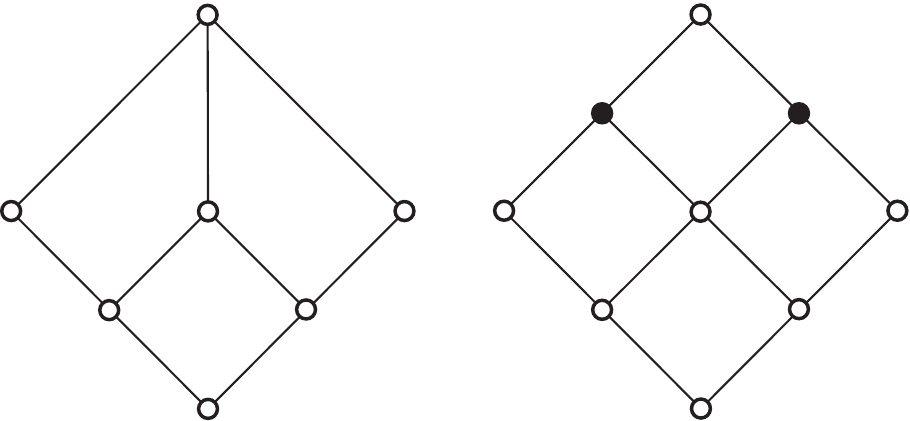}}
\caption{Obtaining the natural diagram of $\seven$}
\label{F:goldenratio}
\end{figure}
 
\subsection*{$\E C_1$-diagrams}%
\index{C1 diagram@${\E C}_1$-diagram}%
\index{Diagram!${\E C}_1$|(}%

This research tool, introduced by G. Cz\'edli in 2017,
has been playing an important role in recent papers
(see G. Cz\'edli \cite{gC17}--\cite{gCcc},
my joint paper with G.~Cz\'edli \cite{CG21}, and also G.~Gr\"atzer~\cite{gG21a};
for the definition, see G.~Cz\'edli \cite{gC17} and G.~Gr\"atzer~\cite{gG21a}).

In the diagram of an SR lattice $K$,\index{Lattice!slim rectangular (SR)}\index{Slim rectangular (SR) lattice}\index{Lattice!SR}\index{SR lattice}
a \emph{normal edge} (\emph{line}) has a slope of $45\degree$ or~$135\degree$.
Any edge (line) of slope strictly between $45\degree$ and $135\degree$ is \emph{steep}.
A \emph{normal edge} (\emph{line}) has a slope of $45\degree$ or $135\degree$.
A \emph{peek sublattice} in $K$ is an $\sseven$ sublattice whose top three does are covering.

\begin{figure}[htb]
\centerline{\includegraphics{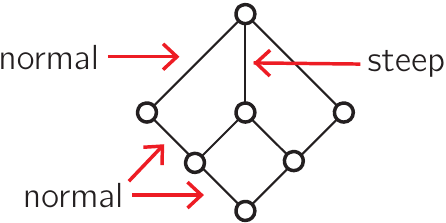}}
\caption{Normal and steep edges}
\label{F:arrows}
\end{figure}

\begin{definition}\index{C1 diagram@${\E C}_1$-diagram}\index{Diagram!${\E C}_1$}
A diagram of an SR lattice $L$ is a \emph{${\E C}_1$-diagram},
if the middle edge of a peek sublattice
is steep and all other edges are normal.
\end{definition}\index{Rectangular!lattice}\index{Lattice!rectangular}

\begin{theorem}\label{T:well}
Every SR lattice $L$ has a ${\E C}_1$-diagram.
\end{theorem}

This was proved in G. Cz\'edli \cite{gC17}.
My note \cite{gG21a} presents a short and direct proof.


It turns out that the two approaches to nice diagrams of SR latices are the same
(see my paper \cite{gG22a}).
 
\begin{theorem}\label{T:C1=natural}
Let $L$ be a SR lattice. 
Then a natural diagram of $L$ is a~${\E C}_1$-diagram. 
Conversely, every ${\E C}_1$-diagram is natural.
\end{theorem}

\subsection{Lamps}\label{S:Lamps}

Let $L$ be a SR lattice. On the edges of $L$ with meet-irreducible bottoms,
the internal swing, $\fp \inswing \fq$, is a reflexive, symmetric, and transitive binary relation.
So these edges are partitioned into equivalence classes; 
G. Cz\'edli \cite{gC21} calls them \emph{lamps}.

The concept of a lamp is a crucial tool for the discovery 
of most deep properties of congruence lattices of SPS lattices.
It is a very deep and technical concept, 
and it will await Cz\'edli to write an introduction to the field in the style of this paper.

G. Cz\'edli~\cite{gCce} also utilizes lamps to find an infinite family 
of independent properties of congruence lattices of SPS lattices. 

There are a number of ways to view how congruences spread in SP lattices. 
The first approach uses trajectories, 
see G. Cz\'edli~\cite{gC14} and my paper \cite{gG18}.
The second is the Swing Lemma. The third is lamps. 
They are very similar in some ways, but quite different technically.

\section{Four major properties}\label{S:major}

This section introduces four major properties of congruences of SR lattices,
as introduced in G. Cz\'edli \cite{gC21}.

Let $K$ be a slim, planar, semimodular  lattice 
with at least three elements and let~$\E P$ be
the ordered set of join-irreducible congruences of $K$.

\subsection{The No Child Property}

See G. Cz\'edli \cite{gC21}, illustrated in Figure~\ref{F:child}.

\begin{theorem}[No Child Property]
Let $x \neq y \in \E P$ and let $z$ be a maximal element of $\E P$. 
Let us assume that both $x$ and $y$ are covered by $z$ in $\E P$. 
Then there is no element $u \in \E P$ such that $u$ is covered by $x$ and $y$.
\end{theorem}

\begin{figure}[htb]
\centerline{\includegraphics{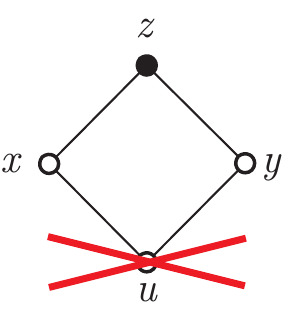}}
\caption{No Child Property}
\label{F:child}
\end{figure} 

\subsection{The Four-crown Two-pendant Property}

See G. Cz\'edli \cite{gC21} and illustrated in Figure~\ref{F:4crown}.

\begin{theorem}[Four-crown Two-pendant Property]
There is no cover-preserving embedding of the ordered set $\E R$ of Figure~\ref{F:4crown}
into $\E P$ satisfying the property that any maximal element of~$\E R$
maps into a~maximal element of $\E P$.
\end{theorem}

\begin{figure}[htb]
\centerline{\includegraphics{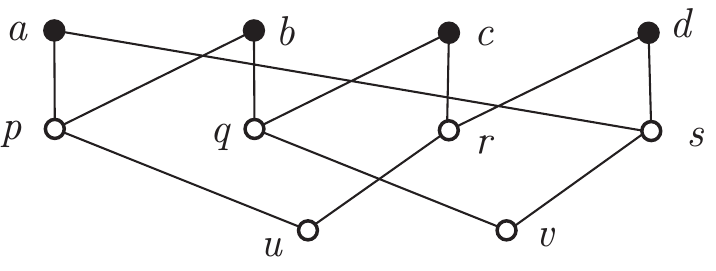}}
\caption{The ordered set $\E R$}
\label{F:4crown}
\end{figure}

\subsection{The Partition Property}

See G. Cz\'edli \cite{gC21}.

\begin{theorem}[Partition Property]
The set of maximal elements of $\E P$ can be partitioned 
into two nonempty subsets such that no two distinct elements 
in the same subset have a common lower cover.
\end{theorem}
 
\subsection{The Maximal Cover Property}
 
 See G. Cz\'edli \cite{gC21}.

\begin{theorem}[Maximal Cover Propery]
If $x \in \E P$ is covered by a maximal element $y \in \E P$, 
then $y$ is not the only cover of $x$.
\end{theorem}
\subsection{Translation}
 “You pays your money and you takes your choice.”
 Aldous Huxley, Brave New World
 
 Cz\'edli and I use different names for some of these properties.  
Translation:\\
\  \\ \ \\
\tbf{Cz\'edli}\q\q\q\q	\q\q\q\q\q\q\q\q\q\ \		\ \	\tbf{Gr\"atzer}\\
\ \\
Forbidden Marriage Property\q\q\q\q\      		\,	No Child Property\\
Dioecious Maximal Elements Property\q                 Maximal Cover Property\\
Bipartite Maximal Elements Property\q	\   		Partition Property\\

 \section{The 3P3C Property}\label{S:3P3C}

The last major property we discuss is in G. Cz\'edli and G. Gr\"atzer~\cite{gG22}. 

\begin{theorem}[Three-pendant Three-crown Property]
The ordered set of Figure~\ref{F:3P3C}
has no cover presering embedding into $\E P$.
\end{theorem}
 \begin{figure}[htb]
\centerline{\includegraphics{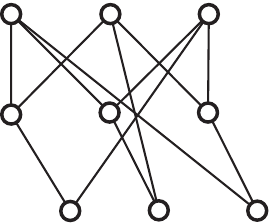}}
\caption{The ordered set for the three-pendant Three-crown Property}
\label{F:3P3C}
\end{figure}

\section{Some proofs}\label{S:proofs}

In this section, I sketch the proofs of the Partition and the
Two-pendant Four-crown Properties of congruences of SR lattices,
based on my paper \cite{gG21}.

\subsection*{The Partition Property}\label{S:partition}\index{Partition Property}

We start with a lemma.

\begin{lemma}\label{L:disjoint}
Let $X $ and $Y$ be distinct edges on the upper-left boundary of~$K$. 
Then there is no edge $Z$ of $K$ such that 
\begin{equation}\label{E:ZXY}
   \col Z \prec \col X, \col Y.
\end{equation}
\end{lemma}

\begin{proof}
By way of contradiction, let $Z$ be an edge such that \eqref{E:ZXY} holds.
Since~$X$ and~$Y$ are on the upper-left boundary, 
there exist normal-up edges $S_X, S_Y$ and steep edges $T_X, T_Y$ such that 
   \[
      X \perspdn S_X \exswing T_X,\q  Y \perspdn S_Y \exswing T_Y,\q  
      Z \in \traj {T_X} \ii \traj {T_Y}.
   \] 
The third formula implies that $T_X = T_Y$. 
Since  $S_X$ and $S_Y$ are normal-up,
it follows that $X = Y$, contradicting the assumption that $X $ and $Y$ are distinct. 
\end{proof}

The set of maximal elements of $\E P$ is the same 
as the set of~colors of edges in the upper boundaries.
The latter we can partition into the set of edges~$\E L$ in~the upper-left boundary
and the set of edges $\E R$ in the upper-right boundary.
Let~$X $ and $Y$ be distinct edges in $\E L$. 
By Lemma~\ref{L:disjoint}, 
there is no edge $Z$ of $K$ such that $\col Z \prec \col X, \col Y$. 
By symmetry, this verifies the Partition Property.\index{Partition Property}

\subsection*{The Two-pendant Four-crown Property}\label{S:Crown}\index{Two-pendant Four-crown Property}

By way of contradiction, 
assume that the ordered set $\E R$ of Figure~\ref{F:4crown}
is a~cover-preserving ordered subset of $\E P$,
where $a,b,c,d$ are maximal elements of~$\E P$. 
So there are edges $A,B,C,D$ on the upper boundary of~$K$, 
so that  $\col A = a$, $\col B = b$, $\col C=c$, $\col D = d$.
By left-right symmetry, we can assume that the edge $A$ 
is on the upper-left boundary of~$K$. Since $p \prec a, b$ in~$\E P$,
it follows from Lemma~\ref{L:disjoint} that the edge $B$
is on the upper-right boundary of $K$, and so is $D$.
Similarly, $C$ is on the upper-left boundary of $K$. 

Because of the automorphisms of the ordered set $\E R$,
it is sufficient to deal with one case only:
$C$ is below $A$ and $B$ is below $D$.

\begin{figure}[htb]
\centerline{\includegraphics[scale=1.2]{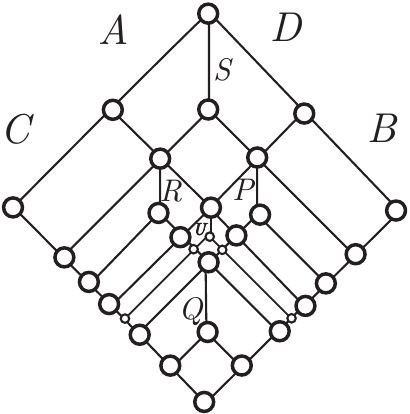}}
\caption{Illustrating the proof of the Two-pendant Four-crown Property}
\label{F:CABDx}
\end{figure}\index{Two-pendant Four-crown Property}

So there is a peak sublattice $\SfS 7$ with middle edge $P$ 
(as in Figure~\ref{F:CABDx})
so that $A$ and $B$ are down-perspective 
to the upper-left edge and the upper-right edge of this peak sublattice, respectively. 
We define, similarly, the edge $Q$ for $C$ and $B$,
the edge $S$ for $A$ and~$D$, the edge $R$ for $C$ and~$D$,
and the edge $U$ for $R$ and~$P$.
Finally, $v \prec q, s$ in $\E R$, 
therefore, there is a peak sublattice~$\SfS 7$ with middle edge~$V$
with upper-left edge $V_l$ and the upper-right edge~$V_r$
so that  $S \perspdn V_l$ and $Q \perspdn V_r$, or symmetrically.
Then $V$ is in both the peak sublattices containing $\SfS 7$, $S$, and $Q$,
respectively, a contradiction.
This concludes the proof of the Two-pendant Four-crown Property.

Of course, the diagram in Figure~\ref{F:CABDx} is only an illustration.
The grid could be much larger, the edges $A, C$ and $B, D$ may not be adjacent, 
and there maybe lots of other elements in $K$.

\section{One more problem}\label{S:more}

Let $D$ be a finite distributive lattice of the form $D^h \dotplus \SB 2$,
the glued sum of a finite distributive lattice $D^h$ with $\SB 2$.
One can think of $D^h$ as $D$ \emph{with the hat removed}.

Recall that a rectangular lattice $L$ is a \emph{patch lattice}, 
if the corners are dual atoms.
It is easy to verify that a rectangular lattice $L$ is a \emph{patch lattice}
if{}f $\Con L$ is a glued sum of a finite distributive lattice $\Con^h L$ with $\SB 2$.

\begin{named}{Patch Problem} 
Characterize the congruence lattice of a patch lattice $L$ 
in terms of the properties of the finite distributive lattice $\Con^h L$.
\end{named}

\end{document}